\documentclass[12pt]{amsart}
\usepackage{amsfonts, amssymb, latexsym, graphicx, hyperref, amsmath}
\usepackage[vcentermath]{youngtab}
\usepackage{ytableau}
\usepackage{tikz}
\usepackage{tikz-cd}
\usepackage{enumitem}
\usepackage{mathtools}
\usepackage[utf8]{inputenc}

\usepackage{hyperref}
\hypersetup{colorlinks=true,linkcolor=blue,citecolor=blue,linktocpage}

\usepackage{algorithm}
\usepackage{algpseudocode}

\usetikzlibrary{calc, shapes, backgrounds,arrows,positioning,decorations.pathmorphing}

\newtheorem{theorem}{Theorem}[section]
\newtheorem{proposition}[theorem]{Proposition}

\newtheorem{lemma}[theorem]{Lemma}

\newtheorem{claim}[theorem]{Claim}

\newtheorem*{claim*}{Claim}
\newtheorem{corollary}[theorem]{Corollary}
\newtheorem{Main Conjecture}[theorem]{Main Conjecture}
\newtheorem{conjecture}[theorem]{Conjecture}
\newtheorem{problem}[theorem]{Problem}

\theoremstyle{definition}
\newtheorem{definition}{Definition}
\theoremstyle{remark}
\newtheorem{example}[theorem]{Example}

\newcommand{\prt}[1]{\langle #1 \rangle}

\title{Approximate counting of standard set-valued tableaux}
\author{Reuven Hodges}
\author{Gidon Orelowitz}
\address{Dept.~of Mathematics, University of Illinois at Urbana-Champaign, Urbana, IL 61801}

\email{rhodges@illinois.edu, gidono2@illinois.edu}

\date{\today}
\begin{document}
\begin{abstract}
We present a randomized algorithm for generating standard set-valued tableaux by extending the Green-Nijenhuis-Wilf hook walk algorithm. In the case of asymptotically rank two partitions, we use this algorithm to give a fully polynomial almost uniform sampler (FPAUS) for standard set-valued tableaux. This FPAUS is then used to construct a fully polynomial randomized approximation scheme (FPRAS) for counting the number of standard set-valued tableaux for such shapes. We also construct a FPAUS and FPRAS for standard set-valued tableaux when either the size of the partition or the difference between the maximum value and the size of the partition is fixed. Our methods build on the work of Jerrum-Valiant-Vazirani and provide a framework for constructing FPAUS's and FPRAS's for other counting problems in algebraic combinatorics.
\end{abstract}
\maketitle

\section{Introduction}
A \emph{partition} is a weakly decreasing sequence of non-negative integers $\lambda = (\lambda_1 \geq \cdots \geq \lambda_k > 0)$. The \emph{Young diagram} of a partition $\lambda$ is a collection of left justified boxes, with $\lambda_i$ boxes in the $i$th row from the top. The \emph{rank} of a partition is the length of the main diagonal in the Young diagram of the partition.

Let $N\geq|\lambda|$. A \emph{semistandard tableau} of shape $\lambda$ is a an assignment of a single value from $1,\ldots,N$ to each box of $\lambda$, such that it is \emph{column standard} (the values increase in each column from top to bottom) and \emph{row standard} (the values increase weakly in each row from left to right). A \emph{standard tableau} is a semistandard tableau where $N = |\lambda|$ and each value in $1,\ldots,N$ appears exactly once. 

An \emph{$N$-semistandard set-valued tableau} of shape $\lambda$ is an assignment of a nonempty subset of the values from $1,\ldots,N$ to each box of $\lambda$, such that if a single value from each box is selected then the result is column and row standard. An \emph{$N$-standard set-valued tableau} is a $N$-semistandard set-valued tableau such that each value from $1,\ldots,N$ appears exactly once. Let ${\sf SVT}(\lambda, N)$ be the set of $N$-standard set-valued tableau of shape $\lambda$ and set \[ f^{\lambda,N} = |{\sf SVT}(\lambda, N)|. \]

Set-valued tableaux were introduced in \cite{B02} by A.~Buch to study the $K$-theory of Grassmannians. As part of this work, he showed that the symmetric Grothendieck polynomial $\mathfrak{G}_{\lambda}$ has a combinatorial interpretation as the generating function for semistandard set-valued tableau. Subsequently, set-valued tableaux have appeared in the literature on poset edge densities~\cite{RTY18,HLL21}, combinatorial formulas for Lascoux polynomials~\cite{MPS18,BSW20}, and in Brill-Noether theory~\cite{CLPT18,CP21}. In the latter setting, the algebraic Euler characteristic of the Brill-Noether space can be expressed in terms of $f^{\lambda,N}$ for $\lambda$ rectangular. 

As observed by C.~Monical, B.~Pankow, and A.~Yong in \cite[Proposition 4.3]{MPY19}, the computation of $f^{\lambda,N}$ is closely related to counting \emph{Hecke words} of length $N$ whose \emph{Demazure product} is a fixed permutation in the \emph{symmetric group}. They also show that there is no algorithm for computing $f^{\lambda,N}$ that is polynomial-time in the bit length of $|\lambda|$ and $N$. This follows from the fact that the output, $f^{\lambda,N}$, is doubly exponential in the bit length of $|\lambda|$ and $N$. In light of this, they ask instead:

\begin{problem}[{\cite[Problem 1.5]{MPY19}}]Does there exist an algorithm to compute $f^{\lambda,N}$ in time polynomial in $|\lambda|$ and $N$.
\label{prob:main}\end{problem} 
We give an answer to the approximation theoretic version of this question for partitions that are contained within the union of a fixed rectangle with a partition of rank two. Such partitions will be referred to as \emph{asymptotically rank two} since in the limit, as $|\lambda|$ and $N$ grow, the combinatorics of these shapes approximate that of partitions of rank two. Finally, we give an approximation theoretic answer to Problem~\ref{prob:main} for set-valued tableaux where either the size of the partition or the difference between the maximum value and the size of the partition is fixed. 

\subsection{Main Results} Our primary result is a randomized polynomial time algorithm that approximates $f^{\lambda,N}$, when $\lambda$ is asymptotically rank two, to within a factor of $\epsilon \in (0,1]$ with high probability. Explicitly, we give \emph{fully polynomial randomized approximation scheme} (FPRAS) for the number of $N$-standard set-valued tableau for such $\lambda$, which computes an approximation $A$ such that
\[
P((1 - \epsilon)f^{\lambda, N} \leq A \leq (1 + \epsilon)f^{\lambda, N}) \geq 1 - \delta
\]
for any $\epsilon, \delta \in (0,1]$ in time polynomial in $|\lambda|$, $N$, $\frac{1}{\epsilon}$, and $\ln \delta^{-1}$. 

\begin{theorem}
\label{thm:mainFPRASresult}
Fix $\mu = (p^q)$ for some $p,q \in \mathbb{N}$. Let $\lambda$ be a partition such that $\lambda \subseteq \mu \cup \lambda^{\circ}$ where $\lambda^{\circ}$ is a partition of rank two. There is a FPRAS for $f^{\lambda, N}$.
\end{theorem}

As a special case we have:

\begin{corollary}
\label{cor:mainFPRASresultspecial}
Let $\lambda$ be a partition of rank less than three. There is a FPRAS for $f^{\lambda, N}$.
\end{corollary}

We also give an FPRAS when some of the input parameters are fixed.

\begin{theorem}
\label{thm:mainFPRASresultfixed}
If $|\lambda|$ or $N-|\lambda|$ is fixed, then there is a FPRAS for $f^{\lambda, N}$.
\end{theorem}

A polynomial time algorithm for computing $f^{\lambda, N}$ exactly for any $\lambda$ with fixed $N-|\lambda|$  is given in \cite[Proposition 4.5]{MPY19}. The authors are not aware of a polynomial time algorithm for computing $f^{\lambda, N}$ exactly for any $N$ with fixed $|\lambda|$. In~\cite{D18}, P.~Drube gives exact formulas, in the case of two row shapes, for the number of $N$-standard set-valued tableaux with fixed \emph{density}, that is, when each box contains a fixed number of entries.

\subsection{Sampling and counting combinatorial objects} These results are achieved by first constructing a randomized algorithm, Algorithm~\ref{svgen} (${\sf SVGen})$, for generating a $N$-standard set-valued tableau. Then ${\sf SVGen}$ is used to bootstrap the Markov chain $\mathcal{MC}_{\sf SVT}$, yielding a \emph{fully polynomial almost uniform sampler} (FPAUS) for $N$-standard set-valued tableau for asymptotically rank two partitions. A FPAUS on a set $\mathcal{S}$ is an algorithm that takes as input a bias parameter $\delta$ and outputs a random $T \in \mathcal{S}$ from a distribution $\Gamma$ on $\mathcal{S}$ such that 
\[
d_{TV}(\Gamma, U) \leq \delta,
\]
where $d_{TV}$ is the \emph{total variance distance} and $U$ is the uniform distribution on $\mathcal{S}$, in time polynomial in the problem size and $\log \delta^{-1}$. We construct a FPAUS for a large class of standard set-valued tableau in in Theorem~\ref{theorem:main}, and by setting $k=N$ in this theorem we recover a FPAUS for asymptotically rank two $\lambda$.

For problems in $\#P$, the existence of a FPAUS in \emph{self-reducible} problems is computationally equivalent to the existence of a FPRAS~\cite{JVV86}. This groundbreaking result, and a later generalization by M.~Dyer and C.~Greenhill~\cite{DG99}, have been used, especially in combination with Markov chain methods, to give FPRAS for many important problems in $\#P$. One of the most successful applications of this methodology is giving a FPRAS for computing the permanent of an arbitrary $n \times n$ matrix with non-negative entries~\cite{JSV04}.

The definition of self-reducibility is technical, and depends strongly on the encoding used for problem instances, and so we will avoid introducing it. The fundamental idea at the core of self-reducibility is that the problem may be expressed as a polynomially bounded (in the problem size) number of sub-problems, each of which are simpler versions of that same problem. We will refer to such a problem as \emph{essentially self-reducible}. We reformulate the computation of $f^{\lambda, N}$ so that it is essentially self-reducible, and then employ the ideas of~\cite{JVV86} to directly construct a FPRAS in Theorem 6.1 which culminates in a proof of Theorem~\ref{thm:mainFPRASresult} and Theorem~\ref{thm:mainFPRASresultfixed}.

It is our belief that our methods provide a useful framework for tackling a multitude of open sampling and counting problems in algebraic combinatorics. Given a randomized polynomial-time algorithm that generates all elements of a set of combinatorial objects with \emph{any} distribution, Corollary 5.2 provides a criterion for converting the algorithm into a FPAUS for that set. Then, so long as the associated counting problem is essentially self-reducible, the FPAUS may be converted into a FPRAS. This essentially self-reducible condition is not particularly restrictive, and many problems, especially those that involve counting fillings of Young diagrams, may be reformulated to become essentially self-reducible. For example, counting semistandard tableaux of a fixed shape and content, i.e. computing \emph{Kostka coefficients}, can be reformulated so that it is essentially self-reducible.

Our paper is organized as follows: Section 2 recalls useful definitions and notations in tableau combinatorics and complexity theory. In Section 3, we define the algorithms that form the building blocks of our FPAUS. In Section 4, we analyze the probability that our algorithm returns a fixed standard set-valued tableau, as well as prove tight bounds on the minimum and maximum probabilities. In Section 5, we convert the algorithm into a FPAUS, and prove that it runs in polynomial time in certain cases.  In Section 6 we show that this FPAUS generates a FPRAS.

\section{Background and notation}
In this section we introduce the definitions and notation that will be used throughout.
\subsection{Set-valued tableaux notation and background} Let ${\sf SSYT}(\lambda, N)$ be the set of semistandard tableau with values in $1,\ldots,N$ and ${\sf SYT}(\lambda)$ be the set of standard tableau of shape $\lambda$.  When $N = |\lambda|$, ${\sf SVT}(\lambda, N) \cong {\sf SYT}(\lambda)$.

For $\mu\subseteq \lambda$, the \emph{skew partition} $\lambda\setminus \mu$ consists of all boxes that are in $\lambda$, but not in $\mu$.  The Young diagram of $\lambda\setminus \mu$ is similarly the Young diagram of $\lambda$, with any boxes in the Young diagram of $\mu$ removed.  ${\sf SYT}(\lambda\setminus \mu)$ is the set of all standard tableau of shape $\lambda\setminus \mu$, which is defined identically to the case of partitions.  

If there is a box in row $r$ and column $c$ of the Young diagram of $\lambda$ we write $(r,c) \in \lambda$. For $T \in {\sf SSYT}(\lambda, N)$ and $(r,c) \in \lambda$, $T(r,c)$ is the value assigned to that box of $T$. For $T \in {\sf SVT}(\lambda, N)$, $T(r,c)$ is the subset of $\{ 1,\ldots,N \}$ assigned to that box of $T$.

For $0 \leq k \leq N$, a \emph{$N\prt{k}$-standard set-valued pre-tableau} of shape $\lambda$ is an assignment of a subset of the values from $k+1,\ldots,N$ to each box of $\lambda$, such that
\begin{itemize}
\item[(i)] each value from $k+1,\ldots,N$ appears exactly once;
\item[(ii)] if $T(r,c) \neq \emptyset$, then $(r+1,c) \in \lambda$ implies $T(r+1,c) \neq \emptyset$ and $(r,c+1) \in \lambda$ implies $T(r,c+1) \neq \emptyset$;
\item[(iii)] if $T(r,c) \neq \emptyset$, then $(r+1,c) \in \lambda$ implies $\max(T(r,c)) < \min(T(r+1,c))$ and $(r,c+1) \in \lambda$ implies $\max(T(r,c)) < \min(T(r,c+1))$;
\item[(iv)] $|\{ (r,c) \in \lambda : T(r,c) = \emptyset\}| \leq k$.
\end{itemize}
Given a $N\prt{k}$-standard set-valued pre-tableau $S$, let ${\sf SVT}(\lambda, N, S) \subseteq {\sf SVT}(\lambda, N)$ be the subset of $N$-standard set-valued tableau $T$ of shape $\lambda$, such that $S \subseteq T$ (that is, $S(r,c) \subseteq T(r,c)$ for $(r,c) \in \lambda$). If the choice of $N$ for the argument is obvious (for example, if $S$ is non-empty), then it is omitted.  It is routine to verify that ${\sf SVT}(\lambda, N, S) > 0$ for any such $S$. We denote the unique $N\prt{N}$-standard set-valued pre-tableau by $E_{\lambda, N}$, with $E_{\lambda, N}(r,c)=\emptyset$ for $(r,c) \in \lambda$. Then ${\sf SVT}(\lambda, N) = {\sf SVT}(\lambda, N, E_{\lambda, N})$. Any $N\prt{0}$-standard set-valued pre-tableau is itself a $N$-standard set-valued tableau.

Let $f^{\lambda} = |{\sf SYT}(\lambda)|$, $f^{\lambda,N} = |{\sf SVT}(\lambda, N)|$, $f^{\lambda,N,S} = |{\sf SVT}(\lambda, N, S)|$, and $f^{\lambda\setminus \mu} = |{\sf SYT}(\lambda\setminus \mu)|$. Let $\mathcal{SVT}$ be the set of all $(\lambda, N, S)$ such that $\lambda$ is a partition, $|\lambda| \leq N \in \mathbb{N}$, and $S$ a $N\prt{k}$-standard set-valued pre-tableau of shape $\lambda$ for some $0 \leq k \leq N$.

\begin{example}
A semistandard tableau for $N=6$, a standard tableau, a $9\prt{5}$-standard set-valued tableau, and $9$-standard set-valued pre-tableau of shape $(3,2)$ are listed below, as well as a standard tableau of shape $(3,2)\setminus (2)$.  The brackets in the set notation are omitted for clarity.
\begin{center}
\ytableausetup{boxsize=1.5em}
\begin{ytableau}
2 & 3 & 3  \\
4 & 6 \\
\end{ytableau}
\qquad \qquad 
\begin{ytableau}
1 & 2 & 4  \\
3 & 5 \\
\end{ytableau}
\qquad \qquad
\begin{ytableau}
\, & \, & 6  \\
\, & {\scriptstyle 7,8,9} \\
\end{ytableau}
\qquad \qquad
\begin{ytableau}
{\scriptstyle 1,2} & 5 & 6  \\
{\scriptstyle 3,4} & {\scriptstyle 7,8,9} \\
\end{ytableau}
\qquad
\begin{ytableau}
\none  & \none  & 2  \\
1 & 3 \\
\end{ytableau}
\end{center}
\end{example}

\subsection{Complexity Theory}
Given two functions $f,g:\mathbb{N}^n_{>0}\to \mathbb{R}_{>0}$, $f=O(g)$ if there exist $c,M_1,\dots M_n>0$ such that $f(x_1,\dots, x_n) \leq cg(x_1,\dots x_n)$ whenever $x_i\geq M_i$ for all $i$.  Say that $f = \Theta(g)$ if $f = O(g)$ and $g = O(f)$.  Define $f = {\sf poly}(x_1,\dots, x_n)$ if there exists some polynomial $h\in \mathbb{R}[x_1,\dots, x_n]$ such that $f = O(h)$.

\section{The generation algorithm}
\label{sec:genAlg}

Fix a partition $\lambda$ and $N \geq |\lambda|$.  In this section we introduce the algorithm SVGen which generates a random $N\prt{k}$-standard set-valued pre-tableau. It will not generate these standard set-valued tableau uniformly at random, but will be the foundation for the FPAUS.
\begin{definition}
The hook of a box $(r,c) \in \lambda$ is \[h_{\lambda}(r,c) = \{ (p,q) \in \lambda : p = r\textit{ with } q > c\textit{, or }q=c\textit{ with } p > r\}.\]
\end{definition}

\begin{definition}
An $(r,c) \in \lambda$ is a \emph{lower right box} if $h_{\lambda}(r,c)=\emptyset$. 
\end{definition}

For $A\subset \mathbb{N}^2$, let $NW(A) = \{(r,c)\in A: \{(r,c)\} = A\cap ([r]\times [c])\}$. Let ${\sf randEl}(A)$ be a function which returns an element of the set $A$, uniformly at random.

\begin{algorithm}[H]
  \caption{Use the hook walk algorithm to return a lower right box of $\lambda$}\label{hookwalk}
  \begin{algorithmic}[1]
    \Require{$\lambda$ is a partition}
    \Statex
    \Procedure{Hook}{$\lambda$}\Comment{Returns a lower right box of $\lambda$}
    \State $(r,c) \gets {\sf randEl}(\{ (p,q) \in \lambda \})$

    \While{$h_{\lambda}(r,c) \neq \emptyset$}
      \State $(r,c) \gets {\sf randEl}(h_{\lambda}(r,c))$
    \EndWhile\label{}
    \State \textbf{return} $(r,c)$\Comment{(r,c) is a lower right box}
    \EndProcedure
  \end{algorithmic}
\end{algorithm}

\begin{algorithm}[H]
  \caption{Generate a random $N$-standard set-valued tableau of shape $\lambda$}\label{svgen}
  \begin{algorithmic}[1]
    \Require{$\lambda$ is a partition, $N \geq |\lambda|$ is a positive integer, $k \leq N$ a nonnegative integer, $T$ a $N\prt{k}$-standard set-valued pre-tableau such that ${\sf SVT}(\lambda, N, T) > 0$}
    \Statex
    \Procedure{SVGen}{$\lambda$, $N$, $k$, $T$}
    \Comment{\textit{Returns an element of ${\sf SVT}(\lambda, N, T)$}}
    \If{k = N}
    \State $(r,c) \gets Hook(\lambda)$
    \State $T(r,c) \gets \{ N \}$
    \State $k \gets 1$
    \EndIf

    \For{$M \gets k \textrm{ to } 1$}
      \State $\lambda' \gets \{ (r,c) \in \lambda : T(r,c) = \emptyset \}$
      \If{${\sf rand}(0,1) \leq (M - |\lambda'|) / M $} \label{svgen:line:mainif}
        \State $(r,c) \gets {\sf randEl}(NW(\{ (p,q) \in \lambda : T(p,q) \neq \emptyset \} ))$
        \State $T(r,c) \gets T(r,c) \cup \{ M \}$

        \Else
        \State $(r,c) \gets Hook(\lambda')$
        \State $T(r,c) \gets \{ M \}$
      \EndIf

    \EndFor

    \State \textbf{return} $T$\Comment{\textit{T a $N$-standard set-valued tableau of shape $\lambda$}}
    \EndProcedure
  \end{algorithmic}
\end{algorithm}


\begin{proposition}
{\sf SVGen}($\lambda$, $N$, $k$, $T$) returns an element of ${\sf SVT}(\lambda, N, T)$.
\end{proposition}
\begin{proof}
This follows from a straight-forward analysis of {\sf SVGen}.
\end{proof}

\section{Probability Analysis}

For $N = |\lambda|$, ${\sf SVGen}(\lambda, N, 0, E_{\lambda, N})$ is precisely the Green-Nijenhuis-Wilf hook walk algorithm for generating standard tableaux of shape $\lambda$. The hook walk algorithm generates the standard tableaux in ${\sf SYT}(\lambda)$ uniformly at random. In the case $N > |\lambda|$ and $\lambda$ has more than one row and column, ${\sf SVGen}(\lambda, N, 0, E_{\lambda, N})$ no longer generates ${\sf SSYT}(\lambda, N)$ uniformly at random. Nonetheless, we are able to analyze the probability that a given $T \in {\sf SVT}(\lambda, N)$ is generated.


Let $\delta^n = (n,n-1,\dots, 1)$ be the \emph{staircase} of height $n$.  Define ${\sf sv}(\lambda) := \max\{k:\delta^k\subseteq \lambda\}$, and $\delta^{{\sf sv}(\lambda)}$ is known as the \emph{Sylvester triangle} of $\lambda$.  For $T\in {\sf SVT}(\lambda, N)$, define ${\sf cell}_T(k)$ to be the unique $(r,c)$ such that $k\in T(r,c)$, and $T^{\prt{k}}$ to be the unique $N\prt{k}$-standard set-valued pre-tableau such that $T\in {\sf SVT}(\lambda, N,T^{\prt{k}})$. 

Let $\mathbb{P}_{\sf SVG}(T,k):= \mathbb{P}(T = {\sf SVGen}(\lambda,N,k,T^{\prt{k}}))$.  Similarly, for $S$ a $N\prt{k}$-standard set-valued pre-tableau of shape $\lambda$ define $\lambda\setminus S := \{(r,c)\in \lambda: S(r,c)=\emptyset\}$, and $(\lambda\setminus S)^+ := (\lambda\setminus S)\cup NW(\{(r,c)\in \lambda: |S(r,c)|>0\})$. Given a fixed $T\in {\sf SVT}(\lambda, N)$, for simplicity of notation we denote $\lambda^{\prt{k}} = \lambda \setminus T^{\prt{k}}$.

\begin{proposition}
\label{prop:probbounds}
Fix a $T\in {\sf SVT}(\lambda, N)$. Then
\begin{equation}
\mathbb{P}_{\sf SVG}(T,N) = \frac{1}{f^\lambda\binom{N-1}{|\lambda|-1}} \prod_{k\not= \max(T({\sf cell}_T(k)))} |{\text NW}(\{(r,c)\in \lambda: k \leq \max(T(r,c)) \})| ^{-1}.
\end{equation}
For $0 \leq k < N$,
\begin{equation}
\mathbb{P}_{\sf SVG}(T,k) = \frac{1}{f^{\lambda^{\prt{k}}}\binom{k}{|\lambda^{\prt{k}}|}} \prod_{\substack{1\leq i \leq k \\ i\not= \max(T({\sf cell}_T(i)))}} |{\text NW}(\{(r,c)\in \lambda: i \leq \max(T(r,c)) \})| ^{-1}
\end{equation}
\end{proposition}
\begin{proof}
We begin by proving a claim regarding hook-insertion.
\begin{claim}
\label{claim:hookprob}
For any corner $(r,c)$ of $\lambda$, $\mathbb{P}((r,c) = {\sf Hook}(\lambda)) = f^{\lambda \setminus \{(r,c)\}}/f^\lambda$.
\end{claim}
\begin{proof}
In their proof of the hook-length formula, Greene-Nijenhuis-Wilf~\cite{GNW79} detail a probabilistic method for generating standard Young tableaux uniformly at random. ${\sf SVGen}(\lambda, N, 0, E_{\lambda, N})$ is identical to their method when $N = |\lambda|$, and thus generates elements of ${\sf SYT}(\lambda)$ uniformly at random with probability $1/f^{\lambda}$.  

If the first corner selected by ${\sf Hook}(\lambda)$ in ${\sf SVGen}(\lambda, |\lambda|, 0, E_{\lambda, |\lambda|})$ is $(r,c)$, then the remaining entries of the filling are recursively determined by ${\sf SVGen}(\lambda, |\lambda|-1, 1, X)$, where $X$ is equal to $E_{\lambda, |\lambda|}$ with the value $|\lambda|$ placed in $(r,c)$. This is equivalent probabilistically to filling the remaining entries via ${\sf SVGen}(\lambda \setminus \{(r,c)\},|\lambda|-1, 0, E_{\lambda \setminus \{(r,c)\}, |\lambda|-1})$. Thus, each output will appear with probability $1/f^{\lambda \setminus \{(r,c)\}}$.  Therefore, $1/f^\lambda = \mathbb{P}((r,c) = {\sf Hook}(\lambda))(1/f^{\lambda \setminus \{(r,c)\}})$, which implies the claim.
\end{proof}

Now we consider the probability that the algorithm inserts the largest values of each cell in a way that would generate $T$. Suppose that 
\[
\{\max(T(r,c)):(r,c)\in \lambda\} = \{ j_1 < j_2 < \cdots < j_{|\lambda|} \}.
\]
For each $m\in [|\lambda|]$, $\lambda^{\prt{j_m}}$ is a partition with $|\lambda^{\prt{j_m}}|=m$. Setting $j_0 = 0$, we have $\lambda^{\prt{j_{m-1}}}\subset \lambda^{{\prt{j_m}}}$ with $\lambda^{{\prt{j_0}}} = \emptyset$ and $\lambda^{\prt{j_{|\lambda|}}} = \lambda$.


For each $m\in [|\lambda^{\prt{k}}|]$, consider the probability that ${\sf SVgen}(\lambda,N,k,T^{\prt{k}})$ places $j_m$ in ${\sf cell}_T(j_m)$.  If {\sf SVgen} line~\ref{svgen:line:mainif} evaluates to true, then $j_m$ will be placed in a cell with a larger value. Thus {\sf SVgen} line~\ref{svgen:line:mainif} must evaluate to false, which occurs with probability $m / j_m$.  Then, by Claim \ref{claim:hookprob}, the probability that $j_m$ is placed in ${\sf cell}_T(j_m)$ by ${\sf Hook}$ is $f^{\lambda^{\prt{j_{m-1}}}}/f^{\lambda^{\prt{j_m}}}$.  In total, the probability that $j_m$ is inserted into ${\sf cell}_T(j_m)$ is
\begin{equation}
\label{eq:svgenprob1}
\frac{m f^{\lambda^{\prt{j_{m-1}}}}}{j_m f^{\lambda^{\prt{j_m}}}}.
\end{equation}
if $m < |\lambda|$, and $\frac{f^{\lambda^{\prt{j_{|\lambda|-1}}}}}{f^{\lambda^{\prt{j_{|\lambda|}}}}} = \frac{f^{\lambda^{\prt{j_{|\lambda|}-1}}}}{f^{\lambda^{\prt{N}}}} = \frac{f^{\lambda^{\prt{N-1}}}}{f^{\lambda}}$ otherwise.

Next, we consider the probability that the algorithm inserts the non-largest values of each cell in a way that would generate $T$. Suppose
\[
[N]\setminus \{\max(T(r,c)):(r,c)\in \lambda\} = \{l_1 < l_2 < \cdots < l_{N-|\lambda|} \}.
\]
For each $m\in [k - |\lambda^{\prt{k}}|]$, consider the probability that ${\sf SVgen}(\lambda,N,k,T^{\prt{k}})$ places $l_m$ in ${\sf cell}_T(l_m)$. If {\sf SVgen} line~\ref{svgen:line:mainif} evaluates to false, then $l_m$ will be the largest value in the cell in which it is placed. Thus {\sf SVgen} line~\ref{svgen:line:mainif} must evaluate to true, which occurs with probability $m / l_m$. Then, $l_m$ is inserted, uniformly at random, into a cell in 
\[
NW(\{(r,c)\in \lambda:l_m \leq \max(T(r,c))\}).
\]
Combining, the probability that $l_m$ is inserted into ${\sf cell}_T(j_m)$ is 
\begin{equation}
\label{eq:svgenprob2}
\frac{m}{l_m}|NW(\{(r,c)\in \lambda:l_m\leq \max(T(r,c))\})|^{-1}.
\end{equation}


If $0 \leq k < N$, then $\{j_m : m \in [|\lambda^{\prt{k}}|]\}$ and $\{l_m : m\in [k-|\lambda^{\prt{k}}|]\}$ partition $[k]$.  This, in combination with \eqref{eq:svgenprob1} and \eqref{eq:svgenprob2}, yields
\begin{align*}
\mathbb{P}_{\sf SVG}(T,k) =&\  \left(\prod_{m=1}^{|\lambda^{\prt{k}}|} \frac{mf^{\lambda^{\prt{j_{m-1}}}}}{j_mf^{\lambda^{\prt{j_m}}}}\right)\left(\prod_{m=1}^{k - |\lambda^{\prt{k}}|} \frac{m}{l_m}|NW(\{(r,c)\in \lambda:l_m\leq \max(T(r,c))\})|^{-1}\right)\\
=&\ \frac{|\lambda^{\prt{k}}|!(k-|\lambda^{\prt{k}}|)!}{k!}\left(\prod_{m=1}^{|\lambda^{\prt{k}}|} \frac{f^{\lambda^{\prt{j_{m-1}}}}}{f^{\lambda^{\prt{j_m}}}}\right) \cdot \\
\, & \qquad \qquad \qquad \qquad \quad \left(\prod_{m=1}^{k-|\lambda^{\prt{k}}|}|NW(\{(r,c)\in \lambda:l_m\leq \max(T(r,c))\})|^{-1}\right) \\
= &\ \frac{1}{f^{\lambda^{\prt{j_{|\lambda^{\prt{k}}|}}}}\binom{k}{|\lambda^{\prt{k}}|}} \prod_{\substack{1\leq i \leq k \\ i \neq \max(T({\sf cell}_T(i)))}} |{\text NW}(\{(r,c)\in \lambda: i \leq \max(T(r,c)) \})| ^{-1}\\
= &\ \frac{1}{f^{\lambda^{\prt{k}}}\binom{k}{|\lambda^{\prt{k}}|}} \prod_{\substack{1\leq i \leq k \\ i \neq \max(T({\sf cell}_T(i)))}} |{\text NW}(\{(r,c)\in \lambda: i \leq \max(T(r,c)) \})|^{-1}.
\end{align*}

On the other hand, if $k=N$, then $\mathbb{P}_{\sf SVG}(T,k)$ is the probability that $N$ gets inserted into ${\sf cell}_T(N)$ of $\lambda$, times the probability that all of the other values are inserted in a way that would generate $T$. So 
\begin{align*}
\mathbb{P}_{\sf SVG}(T,N) =&\ \mathbb{P}({\sf Hook}(\lambda) = {\sf cell}_T(N)) \cdot \mathbb{P}_{\sf SVG}(T,N-1)\\
=&\ \frac{f^{\lambda^{\prt{N-1}}}}{f^{\lambda}}\frac{1}{f^{\lambda^{\prt{N-1}}}\binom{N-1}{|\lambda^{\prt{N-1}}|}} \prod_{\substack{1\leq i \leq N-1 \\ i\neq \max(T({\sf cell}_T(i)))}} |{\text NW}(\{(r,c)\in \lambda: i \leq \max(T(r,c)) \})| ^{-1}\\
=&\ \frac{1}{f^{\lambda}\binom{N-1}{|\lambda|-1}} \prod_{\substack{1\leq i \leq N-1 \\ i \neq \max(T({\sf cell}_T(i)))}} |{\text NW}(\{(r,c)\in \lambda: i \leq \max(T(r,c)) \})|^{-1},
\end{align*}
which completes the proof.
\end{proof}
\begin{corollary}
Fix a $T\in {\sf SVT}(\lambda, N)$. Then
\begin{equation}
\label{eq:0bounds}
\frac{1}{f^\lambda \binom{N-1}{|\lambda|-1}({\sf sv}(\lambda))^{N-|\lambda|}}\leq P_{\sf SVG}(T,N)\leq \frac{1}{f^\lambda \binom{N-1}{|\lambda|-1}}
\end{equation}
and for all $0 \leq k < N$, 
\begin{equation}
\label{eq:kbounds}
\frac{1}{f^{\lambda^{\prt{k}}}\binom{k}{|\lambda^{\prt{k}}|}({\sf sv}((\lambda^{\prt{k}})^+))^{k-|\lambda^{\prt{k}}|}}\leq P_{\sf SVG}(T,k)\leq \frac{1}{f^{\lambda^{\prt{k}}}\binom{k}{|\lambda^{\prt{k}}|}}.
\end{equation}
These bounds are tight for any $N,\lambda,k$.
\end{corollary}
\begin{proof}
To prove the upper bounds in \eqref{eq:0bounds} and \eqref{eq:kbounds}, observe that there exists a cell whose maximum label is $N$. Thus, $\{(r,c)\in \lambda: i \leq \max(T(r,c)) \}\not= \emptyset$ for $1 \leq i \leq N$. This implies $|{\text NW}(\{(r,c)\in \lambda: i \leq \max(T(r,c)) \})|\geq 1$ for each $i$, and the inequality follows.

To see that this upper bound is tight, consider the unique $T\in {\sf SVT}(\lambda, N)$ such that $|T(r,c)|=1$ for all $(r,c)\not=(1,1)$.  For this $T$, $i \neq \max(T({\sf cell}_T(k)))$ implies $i\in [N-|\lambda|]$. For each $i\in [N-|\lambda|]$, $\{(r,c)\in \lambda: i \leq \max(T(r,c)) \} = \{(1,1)\}$, and hence $|{\text NW}(\{(r,c)\in \lambda: i \leq \max(T(r,c)) \})|= 1$. This yields the desired equality for the upper bound.

\begin{claim}
\label{claim:indepbound}
If $A\subset \mathbb{N}^2$ such that $A\subseteq \lambda$ and no element of $A$ is weakly northwest of any other element, then $|A|\leq {\sf sv}(\lambda)$.
\end{claim}

\begin{proof}
By the definition of ${\sf sv}(\lambda)$, there exists $t\in [{\sf sv}(\lambda)+1]$ such that $(t,{\sf sv}(\lambda)+2-t)\not\in \lambda$.  Fix such a $t$. Then, $\lambda\subset \{(r,c):r<t\}\cup \{(r,c):c<{\sf sv}(\lambda) + 2-t\}$.  By hypothesis, $|A\cap \{(r,c):r<t\}|\leq t-1$, and $|A\cap \{(r,c):c<{\sf sv}(\lambda) + 2-t\}| \leq {\sf sv}(\lambda)+1-t$. Hence $|A| = |A\cap \lambda| \leq |A\cap \{(r,c):r<t\}|+ |A\cap \{(r,c):c<{\sf sv}(\lambda) + 2-t\}| \leq {\sf sv}(\lambda)$.
\end{proof}

To prove the lower bounds in \eqref{eq:0bounds} and \eqref{eq:kbounds}, it is sufficient to show that, for all $k\geq 0$ and $i\in [k]$ such that $i \neq \max(T({\sf cell}_T(i)))$, 
\[
{\sf sv}((\lambda^{\prt{k}})^+)\geq |NW(\{(r,c)\in \lambda: i \leq \max(T(r,c)) \})|.  
\]
Observe that $NW(\{(r,c)\in \lambda: i \leq \max(T(r,c)) \} \subset (\lambda^{\prt{k}})^+$ and no element of $NW(\{(r,c)\in \lambda: i \leq \max(T(r,c)) \}$ is weakly northwest of any other element. Thus, Claim \ref{claim:indepbound} gives us our lower bound.

To show the lower bounds in \eqref{eq:0bounds} and \eqref{eq:kbounds} are tight, let $k\geq 0$, $T \in {\sf SVT}(\lambda, N)$ such that
\begin{itemize}
\item[(i)] $|T(r,c)|=1$ for all $(r,c)\in \lambda^{\prt{k}}\setminus \{({\sf sv}((\lambda^{\prt{k}})^+),1)\}$,
\item[(ii)] ${\sf cell}_T(i)\in \delta^{{\sf sv}((\lambda^{\prt{k}})^+)-1}$ for all $i\in |\delta^{{\sf sv}((\lambda^{\prt{k}})^+)-1}|$,
\item[(iii)] $T({\sf sv}((\lambda^{\prt{k}})^+),1) \supseteq [k-|\lambda^{\prt{k}}|+\binom{{\sf sv}((\lambda^{\prt{k}})^+)}{2}+|\{({\sf sv}((\lambda^{\prt{k}})^+),1)\}\cap \lambda^{\prt{k}}|]\setminus [|\delta^{{\sf sv}((\lambda^{\prt{k}})^+)-1}|]$.
\end{itemize} 
Such a $T$ can be constructed by starting with any $N\prt{k}$-standard pre-tableau of shape $\lambda$. The remaining values are placed, first satisfying (ii), by placing a single value in each cell in $\delta^{{\sf sv}((\lambda^{\prt{k}})^+)-1}$ in any way that does not violate standardness. Then values are placed in cell $({\sf sv}((\lambda^{\prt{k}})^+),1)$ satisfying (iii). Then the remaining empty cells of $\lambda^{\prt{k}}$ are filled with a single value in any way that does not violate standardness.

For such a $T$, $i\not= \max(T({\sf cell}_T(i)))$ and $i\in [k]$ implies that 
\[
|\delta^{{\sf sv}((\lambda^{\prt{k}})^+)-1}| < i \leq k-|\lambda^{\prt{k}}|+|\delta^{{\sf sv}((\lambda^{\prt{k}})^+)-1}|,
\]
and
\begin{equation*}
\{(r,c)\in \lambda: i \leq \max(T(r,c)) \} = \lambda\setminus \delta^{{\sf sv}((\lambda^{\prt{k}})^+)-1} .
\end{equation*}
Thus
\begin{equation*}
NW(\{(r,c)\in \lambda: i \leq \max(T(r,c)) \})= \{(t,{\sf sv}((\lambda^{\prt{k}})^+)-t+1):t\in [{\sf sv}((\lambda^{\prt{k}})^+)]\},
\end{equation*}
which implies
\begin{equation*}
|NW(\{(r,c)\in \lambda: i \leq \max(T(r,c)) \})|= {\sf sv}((\lambda^{\prt{k}})^+),
\end{equation*}
resulting in equality for the lower bound.  
\end{proof}

\section{A FPAUS via Markov chains}
\label{sec:markovAlg}

Our goal in this section is to convert the ${\sf SVGen}$ algorithm presented in Section~\ref{sec:genAlg} into a FPAUS for $N$-standard set-valued tableau. This can be achieved in a number of ways, but the language of Markov chains is particularly convenient for our purposes.

\subsection{Background on Markov chains}
In this section we introduce the required background on Markov chains, following the material and notation of \cite{LP17}.

Let $\mathcal{M}$ be a \emph{Markov chain} with \emph{state space} $\mathcal{X}$ and \emph{transition matrix} $P$. For $x,y \in \mathcal{X}$, $P(x,y)$ is the probability of proceeding from $x$ to $y$, while $P^{t}(x,y)$ is the probability of proceeding from $x$ to $y$ in exactly $t$ steps. $\mathcal{M}$ is \emph{irreducible} if for every $x,y \in \mathcal{X}$, there exists a $t \geq 0$ such that  $P^t(x,y) > 0$, that is, it is possible to eventually transition between every pair of states. The \emph{period} of $x \in \mathcal{X}$ is the greatest common divisor of $\{ t\geq 1 : P^t(x,x) > 0 \}$. A Markov chain is \emph{aperiodic} if all states in $\mathcal{X}$ have period $1$. 

Given two distributions $\mu$ and $\nu$ on $\mathcal{X}$, the \emph{total variance distance} is \[d_{\sf TV}(\mu,\nu) = \max_{A \subseteq \mathcal{X}} |\mu(A) - \nu(A)|.\] A fundamental result in the theory of Markov chains states that any irreducible and aperiodic Markov chain has a unique stationary distribution $\pi$ over $\mathcal{X}$. Explicitly, for all $x,y \in \mathcal{X}$ 
\[ 
\lim_{t \to \infty} P^t(x,y) = \pi(y),
\]
and hence for any $\epsilon > 0$, there exists a $t$ such that  $d_{\sf TV}(P^t(x,\cdot),\pi) \leq \epsilon$. 
A Markov chain is \emph{reversible} if there exists a distribution $\pi$ on $\mathcal{X}$ such that for all $x,y \in \mathcal{X}$,
\[
\pi(x)P(x,y) = \pi(y)P(y,x).
\]
If the chain is reversible, then $\pi$ is the steady state distribution.

Let $d(t) = \max_{x \in \mathcal{X}} d_{\sf TV}(P^t(x,\cdot),\pi)$. The \emph{mixing time} of a chain is \[t_{\sf mix}(\epsilon) = \min\{ t : d(t) \leq \epsilon \}.\] A chain is \emph{rapidly mixing} if $t_{\sf mix}(\epsilon)$ is bounded by a polynomial in $\epsilon^{-1}$ and $n$, where $n$ is a parameter measuring the problem size.

\subsection{The Metropolis algorithm} The Metropolis algorithm takes as input an irreducible Markov chain with state space $\mathcal{X}$ and arbitrary transition matrix $\Psi$, and modifies the chain so that it has stationary distribution $\pi$. 

A new transition matrix is generated as follows. When at a state $x$, use $\Psi(x,\cdot)$ to generate a new state $y$. The chain moves to $y$ with probability $\min\{1, \frac{\pi(y)\Psi(y,x)}{\pi(x)\Psi(x,y)} \}$, and remains at $x$ otherwise. The new transition matrix is
\[
P(x,y) = \begin{cases}\Psi(x,y)\min\{1, \frac{\pi(y)\Psi(y,x)}{\pi(x)\Psi(x,y)} \} & y \neq x \\ 1 - \displaystyle \sum_{\substack{z \in \mathcal{X} \\ x \neq z}} \Psi(x,z)\min\{1, \frac{\pi(z)\Psi(z,x)}{\pi(x)\Psi(x,z)} \} & y = x \\ \end{cases}
\]
It is an easy exercise to verify that this chain is reversible with steady state distribution $\pi$.

\subsection{Bounding mixing time via conductance} Let $P$ be the transition matrix for an irreducible and aperiodic Markov chain with stationary distribution $\pi$. The \emph{conductance} of a Markov chain is a measure of the connectedness of the state space. Formally, the conductance (also known as the \emph{Cheeger constant}) of the chain is
\[ 
\Phi = \displaystyle \min_{\substack{S \subseteq \mathcal{X} \\ \pi(S) \leq \frac{1}{2}}} \frac{\displaystyle \sum_{x \in S, y \in S^c}\pi(x)P(x,y)}{\pi(S)}.
\]
Conductance, combined with results relating mixing time and the eigenvalues of $P$, yields bounds on the mixing time of the chain~\cite{S93},
\[
\frac{1-\Phi}{2\Phi}\ln \frac{1}{\epsilon} \leq t_{\sf mix}(\epsilon) \leq \frac{1}{\Phi^2} \left( \ln \frac{1}{\displaystyle \min_{x \in \mathcal{X}}\pi(x)} + \ln \frac{1}{\epsilon} \right).
\]
An immediate corollary is that conductance completely characterizes rapid mixing.
\begin{corollary}
\label{cor:conductance}
A family of irreducible and aperiodic Markov chains $\mathcal{M}_n$ of problem size $n$ and conductance $\Phi_n$ is rapidly mixing if and only if 
\[
\Phi_n \geq \frac{1}{f(n)}
\]
for some polynomial $f(n)$.
\end{corollary}

\begin{corollary}
\label{cor:conduct2}
Let $\mathcal{M}_n$ be a family of irreducible and aperiodic Markov chains $\mathcal{M}_n$ of problem size $n$ such that there exists a function $\Psi(x)$ such that $\Psi(x) = \Psi(y,x)$ for all $x\not=y$, and $\pi(x) = 1/|\mathcal{X}|$ for all $x$.  Then $\mathcal{M}_n$ is rapidly mixing if and only if \[\frac{1}{|\mathcal{X}|\min_{x\in \mathcal{X}} \Psi(x)} = {\sf poly}(n).\]
\end{corollary}
\begin{proof}
Label the elements of $\mathcal{X}$ as $x_1,\dots, x_{|\mathcal{X}|}$ such that $\Psi(x_1)\leq \Psi(x_2)\leq \dots \leq \Psi(x_{|\mathcal{X}|})$.  

Using the statement hypotheses, we have that 
\begin{align*}
\min_{\substack{S \subseteq \mathcal{X} \\ \pi(S) \leq \frac{1}{2}}} \frac{\displaystyle \sum_{x \in S, y \in S^c}\pi(x)P(x,y)}{\pi(S)} =&\ \min_{1\leq k \leq |\mathcal{X}|/2}\min_{\substack{S \subseteq \mathcal{X} \\ |S| = k}} \frac{\displaystyle \sum_{x \in S, y \in S^c}\pi(x)P(x,y)}{\pi(S)}\\
=&\ \min_{1\leq k \leq |\mathcal{X}|/2}\min_{\substack{S \subseteq \mathcal{X} \\ |S| = k}} \frac{\displaystyle \sum_{x \in S, y \in S^c}\frac{1}{|\mathcal{X}|}P(x,y)}{\frac{k}{|\mathcal{X}|}}\\
=&\ \min_{1\leq k \leq |\mathcal{X}|/2}\frac{1}{k}\min_{\substack{S \subseteq \mathcal{X} \\ |S| = k}} \displaystyle \sum_{x \in S, y \in S^c}P(x,y)\\
=&\ \min_{1\leq k \leq |\mathcal{X}|/2}\frac{1}{k}\min_{\substack{S \subseteq \mathcal{X} \\ |S| = k}} \displaystyle \sum_{x \in S, y \in S^c}\Psi(x,y)\min\{1, \frac{\pi(y)\Psi(y,x)}{\pi(x)\Psi(x,y)} \}\\
=&\ \min_{1\leq k \leq |\mathcal{X}|/2}\frac{1}{k}\min_{\substack{S \subseteq \mathcal{X} \\ |S| = k}} \displaystyle \sum_{x \in S, y \in S^c}\min\{\Psi(x,y), \Psi(y,x) \}\\
=&\ \min_{1\leq k \leq |\mathcal{X}|/2}\frac{1}{k}\min_{\substack{S \subseteq \mathcal{X} \\ |S| = k}} \displaystyle \sum_{x \in S, y \in S^c}\min\{\Psi(y), \Psi(x) \}.
\end{align*}
The summation $\sum_{x \in S, y \in S^c}\min\{\Psi(y), \Psi(x) \}$ has $k(|\mathcal{X}|-k)$ summands. For a fixed $x\in \mathcal{X}$, $\Psi(x)$ can occur as a summand at most $|\mathcal{X}|-k$ times. Combining these two facts, we see that the summation would be minimized if each $\Psi(x_i)$, for $1 \leq i \leq k$, appeared as a summand $(|\mathcal{X}|-k)$ times (since these are the smallest possible summand values). Thus $\sum_{i=1}^k (|\mathcal{X}|-k)\Psi(x_i) \leq \sum_{x \in S, y \in S^c}\min\{\Psi(y), \Psi(x) \}$. This lower bound is attained by setting $S = \{x_1,\dots, x_k\}$, and so this $S$ is the minimum. This implies
\begin{align*}
\min_{\substack{S \subseteq \mathcal{X} \\ \pi(S) \leq \frac{1}{2}}} \frac{\displaystyle \sum_{x \in S, y \in S^c}\pi(x)P(x,y)}{\pi(S)} =&\ \min_{1\leq k \leq |\mathcal{X}|/2}\frac{1}{k}\min_{\substack{S \subseteq \mathcal{X} \\ |S| = k}} \displaystyle \sum_{x \in S, y \in S^c}\min\{\Psi(y), \Psi(x) \}\\
=&\ \min_{1\leq k \leq |\mathcal{X}|/2}\frac{1}{k}(|\mathcal{X}|-k)\sum_{i=1}^k \Psi(x_i).
\end{align*}
Replacing the $k$ in $(|\mathcal{X}|-k)$ with either $0$ or $\frac{|\mathcal{X}|}{2}$ we arrive at the bounds
\[\frac{|\mathcal{X}|}{2}\min_{1\leq k\leq |\mathcal{X}|/2}\frac{1}{k}\sum_{i=1}^k \Psi(x_i) \leq \Phi_n\leq |\mathcal{X}|\min_{1\leq k\leq |\mathcal{X}|/2}\frac{1}{k}\sum_{i=1}^k \Psi(x_i).\]
Since $\min_{1\leq k\leq |\mathcal{X}|/2}\frac{1}{k}\sum_{i=1}^k \Psi(x_i) = \Psi(x_1)$,
\[\Phi_n = \Theta(|\mathcal{X}| \Psi(x_1))\]
Therefore, $\Phi_n$ and $|\mathcal{X}|\min_{x\in \mathcal{X}} \Psi(x)$ have the same growth rate, and in particular their reciprocals have the same growth rate. Thus, if either is bounded by a polynomial in $n$, then the other is bounded by a polynomial in $n$. Our result follows by Corollary \ref{cor:conductance}.
\end{proof}

\subsection{The FPAUS}  Denote by $U$ the uniform distribution on the set ${\sf SVT}(\lambda, N, S)$. A FPAUS for $N$-standard set-valued tableau is a randomized algorithm that takes as input $k$, $\lambda$, and $S$, as well as a bias parameter $\delta$ and outputs a random $T \in {\sf SVT}(\lambda, N, S)$ from a distribution $\Gamma$ on ${\sf SVT}(\lambda, N, S)$, with $d_{\sf TV}(\Gamma,U) \leq \delta$, in time polynomial in $k$,$|\lambda \setminus S|$, $\log \delta^{-1}$. When $S = E_{\lambda, N}$, then $k=N$, $\lambda \setminus S = \lambda$ and this gives a distribution $\Gamma$ on ${\sf SVT}(\lambda, N)$.

Let $\mathbb{P}_{\sf SVG}$ be the distribution on ${\sf SVT}(\lambda, N, S)$ given by $\mathbb{P}_{\sf SVG}(T) = \mathbb{P}(T = {\sf SVGen}(\lambda, N, S))$. Our initial Markov chain $\tilde{\mathcal{MC}}_{\sf SVT}(\lambda, N, S)$ on the state space ${\sf SVT}(\lambda, N, S)$ has transition matrix $\tilde{P}$ given by $\tilde{P}(T,\cdot) = \mathbb{P}_{\sf SVG}(\cdot, k)$. In other words, the distribution at $T$ does not depend on $T$, and hence $\tilde{P}$ is a rank $1$ matrix. We now apply the Metropolis algorithm to modify the steady state distribution of $\tilde{\mathcal{MC}}_{\sf SVT}(\lambda, N, S)$ to the uniform distribution $U$. Since $U(T) / U(F) = 1$ for all $F,T \in {\sf SVT}(\lambda, N, S)$, the new Markov chain, which we denote $\mathcal{MC}_{\sf SVT}(\lambda, N, S)$, will have transition matrix
\[
P(T,F) = \begin{cases}\mathbb{P}_{\sf SVG}(F, k)\min\{1, \frac{\mathbb{P}_{\sf SVG}(T, k)}{\mathbb{P}_{\sf SVG}(F, k)} \} & F \neq T \\ 1 - \displaystyle \sum_{\substack{R \in {\sf SVT}(\lambda, N, S) \\ R \neq T}} \mathbb{P}_{\sf SVG}(R, k)\min\{1, \frac{\mathbb{P}_{\sf SVG}(T, k)}{\mathbb{P}_{\sf SVG}(R, k)} \} & F = T \\ \end{cases}
\]
and steady state distribution $U$.
\begin{lemma}
$\mathcal{MC}_{\sf SVT}(\lambda, N, S)$ is irreducible, aperiodic, and reversible.
\end{lemma}
\begin{proof}
Since $\mathbb{P}_{\sf SVG}(F, k) > 0$ for all $T \in {\sf SVT}(\lambda, N, S)$, it is possible to transition from any state $T$ to any state $F$. This immediately implies that both $\tilde{\mathcal{MC}}_{\sf SVT}(\lambda, N, S)$ and $\tilde{\mathcal{MC}}_{\sf SVT}(\lambda, N, S)$ are irreducible and aperiodic. $\mathcal{MC}_{\sf SVT}(\lambda, N, S)$ is reversible by construction as it is the result of applying the Metropolis algorithm to $\tilde{\mathcal{MC}}_{\sf SVT}(\lambda, N, S)$.
\end{proof}

\begin{theorem}
\label{theorem:main}
Fix a rectangle $\mu = (p^q)$.  Let $\mathcal{F}_{p,q} \subseteq \mathcal{SVT}$ be the subset such that $\lambda \subseteq \mu \cup \lambda^{\circ}$ where $\lambda^{\circ}$ is a partition of rank less than three. Then $\mathcal{MC}_{\sf SVT}(\lambda,N,S)$ is rapidly mixing for all $(\lambda,N,S)\in \mathcal{F}_{p,q}$ and hence there is a FPAUS for $\mathcal{F}_{p,q}$.
\end{theorem}
\begin{proof}
Without loss of generality, we may assume that $p,q\geq 2$ since $\mathcal{F}_{p,q} \subseteq \mathcal{F}_{p',q'}$ for all $p' \geq p$ and $q' \geq q$.  Fix a $T\in {\sf SVT}(\lambda, N,S)$. Recall that $\lambda^{\prt{k}} := \lambda \setminus T^{\prt{k}} = \lambda \setminus S$. By Corollary \ref{cor:conduct2}, it suffices to show that 
\begin{equation}
\label{eqn:polymix00}
\left(|{\sf SVT}(\lambda, N,S)|\min_{T\in {\sf SVT}(\lambda, N,S)}\mathbb{P}_{\sf SVG}(T, k)\right)^{-1} = {\sf poly}(|\lambda^{\prt{k}}|,k).
\end{equation}
By Proposition \ref{prop:probbounds}, \eqref{eqn:polymix00} is equivalent, if $k=N$, to showing that
\begin{equation}
\label{eqn:polymix0}
\left(\frac{|{\sf SVT}(\lambda, N,S)|}{f^\lambda \binom{N-1}{|\lambda|-1}({\sf sv}(\lambda))^{N-|\lambda|}}\right)^{-1} =\frac{f^\lambda \binom{N-1}{|\lambda|-1}({\sf sv}(\lambda))^{N-|\lambda|}}{|{\sf SVT}(\lambda, N)|} = {\sf poly}(|\lambda|,N),
\end{equation}
and if $0 \leq k < N$, that
\begin{equation}
\label{eqn:polymixk}
\left(\frac{|{\sf SVT}(\lambda, N,S)|}{f^{\lambda^{\prt{k}}} \binom{k}{|\lambda^{\prt{k}}|}({\sf sv}((\lambda^{\prt{k}})^+))^{k-|\lambda^{\prt{k}}|}}\right)^{-1} =\frac{f^{\lambda^{\prt{k}}} \binom{k}{|\lambda^{\prt{k}}|}({\sf sv}((\lambda^{\prt{k}})^+))^{k-|\lambda^{\prt{k}}|}}{|{\sf SVT}(\lambda, N,S)|} = {\sf poly}(|\lambda^{\prt{k}}|,k).
\end{equation}
In either case, it is equivalent to show that 
\begin{equation}
\label{eqn:polymixgen}
\frac{f^{\lambda^{\prt{k}}} \binom{k}{|\lambda^{\prt{k}}|}({\sf sv}((\lambda^{\prt{k}})^+))^{k-|\lambda^{\prt{k}}|}}{|{\sf SVT}(\lambda, N,S)|} = {\sf poly}(|\lambda^{\prt{k}}|,k),
\end{equation}
since \eqref{eqn:polymixgen} is equivalent to \eqref{eqn:polymixk}, and differs from \eqref{eqn:polymix0} by a polynomial factor, $\frac{N}{|\lambda|}$.

Our goal is to lower-bound $|{\sf SVT}(\lambda,N,S)|$ by constructing elements $F\in {\sf SVT}(\lambda,N,S)$ as follows.  We break the construction into three cases.

\noindent {\emph Case 1 ($ (\lambda^{\prt{k}})_1 > (\lambda^{\prt{k}})_2 > p$ and $ (\lambda^{\prt{k}})'_1 > (\lambda^{\prt{k}})'_2 > q$):} Let \[\mu^2 = \{(r,c)\in \lambda^{\prt{k}}\setminus \delta^{{\sf sv}((\lambda^{\prt{k}})^+)-1}: r,c>2\}\cup \{(1,(\lambda^{\prt{k}})_1),(\ell(\lambda^{\prt{k}}),1)\}\] and \[\mu^1 = (((p+1)^{q+1})\cap \lambda^{\prt{k}})\setminus \mu^2.\]  Let $D(\lambda^{\prt{k}},\mu) \subseteq {\sf SYT}(\lambda^{\prt{k}})$ be all elements $F_0$ such that ${\sf cell}_{F_0}(i) \in  \mu^1$ for all $i\in [|\mu^1|]$, and ${\sf cell}_{F_0}(i)\in \mu^2$ for all $i\in [|\lambda^{\prt{k}}|-|\mu^2|+1  , |\lambda^{\prt{k}}|]$.

Now, let $R \subset [k]$ such that $|R| = |\lambda^{\prt{k}}|$ and $[|\mu^1|]\cup [k-|\mu^2|+1  , k] \subseteq R$.  Partition $[k]\setminus R$ into $R_1,R_2,\dots , R_{{\sf sv}((\lambda^{\prt{k}})^+)}$.  For an $F_0 \in D(\lambda^{\prt{k}},\mu)$, $F$ is constructed as follows.
\begin{enumerate}
\item
Start with $S$
\item
Place the $i^{th}$ smallest element of $R$ in ${\sf cell}_{F_0}(i)$.
\item
Starting with the smallest value and moving in increasing order, insert the values of $R_1$ into the highest box in the first column such that the inserted element is not the largest element of its cell.
\item
Starting with the largest value and moving in decreasing order, insert the values of $R_2$ into the lowest box of the second column such that the inserted element is not the smallest element of its cell.
\item
For $3\leq i\leq {\sf sv}((\lambda^{\prt{k}})^+)-2$, insert all of the elements of $R_i$ into $({\sf sv}((\lambda^{\prt{k}})^+)+1-i,i)$.
\item
Starting with the largest value and moving in decreasing order, insert the values of $R_{{\sf sv}((\lambda^{\prt{k}})^+)-1}$ into the rightmost box of the second row such that the inserted element is not the smallest element of its cell.
\item
Starting with the smallest value and moving in increasing order, insert the values of $R_{{\sf sv}((\lambda^{\prt{k}})^+)}$ into the leftmost box in the first row such that the inserted element is not the largest element of its cell.
\end{enumerate}

After each step, $F$ remains row and column standard and hence $F \in {\sf SVT}(\lambda,N,S)$. It is an easy check to verify that each choice of $F_0$, $R$, $R_1,R_2,\dots ,$ and $R_{{\sf sv}((\lambda^{\prt{k}})^+)}$ yields a unique $F$. Then $d(\lambda^{\prt{k}},\mu)=|D(\lambda^{\prt{k}},\mu)|$ is the number of possible choices for $F_0$.  There are $\binom{k-|\mu^1|-|\mu^2|}{|\lambda^{\prt{k}}|-|\mu^1|-|\mu^2|}$ choices for $R$ and $({\sf sv}((\lambda^{\prt{k}})^+))^{k-|\lambda^{\prt{k}}|}$ choices for $R_1,\dots, R_{{\sf sv}((\lambda^{\prt{k}})^+)}$.  As a result, \[d(\lambda^{\prt{k}},\mu)\binom{k-|\mu^1|-|\mu^2|}{|\lambda^{\prt{k}}|-|\mu^1|-|\mu^2|}({\sf sv}((\lambda^{\prt{k}})^+))^{k-|\lambda^{\prt{k}}|} \leq |{\sf SVT}(\lambda,N,S)|.\]  Thus, the left hand side of \eqref{eqn:polymixgen} becomes: 
\begin{align*}
\frac{f^{\lambda^{\prt{k}}} \binom{k}{|\lambda^{\prt{k}}|}({\sf sv}((\lambda^{\prt{k}})^+))^{k-|\lambda^{\prt{k}}|}}{|{\sf SVT}(\lambda, N,S)|} &\leq  \frac{f^{\lambda^{\prt{k}}} \binom{k}{|\lambda^{\prt{k}}|}({\sf sv}((\lambda^{\prt{k}})^+))^{k-|\lambda^{\prt{k}}|}}{d(\lambda^{\prt{k}},\mu)\binom{k-|\mu^1|-|\mu^2|}{|\lambda^{\prt{k}}|-|\mu^1|-|\mu^2|}({\sf sv}((\lambda^{\prt{k}})^+))^{k-|\lambda^{\prt{k}}|}}\\
&= \frac{f^{\lambda^{\prt{k}}} \binom{k}{|\lambda^{\prt{k}}|}}{d(\lambda^{\prt{k}},\mu)\binom{k-|\mu^1|-|\mu^2|}{|\lambda^{\prt{k}}|-|\mu^1|-|\mu^2|}}\\
&= \frac{f^{\lambda^{\prt{k}}} }{d(\lambda^{\prt{k}},\mu)} O(\frac{k^{|\mu^1|+|\mu^2|}}{|\lambda^{\prt{k}}|^{|\mu^1|+|\mu^2|}})\\
&= \frac{f^{\lambda^{\prt{k}}} }{d(\lambda^{\prt{k}},\mu)} {\sf poly}(|\lambda^{\prt{k}}|,k),
\end{align*}
where in the last step we have used the fact that $|\mu^1|+|\mu^2|\leq |\mu|+6 = \Theta(1)$.  Thus the following claim proves our desired result in this case.
\begin{claim}
\[\frac{f^{\lambda^{\prt{k}}} }{d(\lambda^{\prt{k}},\mu)} = {\sf poly}(|\lambda^{\prt{k}}|,k)\]
\end{claim} 
\begin{proof}
By the hook-length formula,
\[f^{\lambda^{\prt{k}}} = \frac{|\lambda^{\prt{k}}|!}{\prod_{(r,c)\in \lambda^{\prt{k}}} |h_{\lambda^{\prt{k}}}(r,c)| +1 }.\]

We lower bound $d(\lambda^{\prt{k}},\mu)$ by constructing elements $F \in D(\lambda^{\prt{k}},\mu)$. First, fill the boxes of $F$ in $\mu^1$ with the values $[|\mu^1|]$ in any way that is row and column standard. Second, fill the boxes of $F$ in $\mu^2$ with the values in $[|\lambda^{\prt{k}}|-|\mu^2|+1, |\lambda^{\prt{k}}|]$ in any way that is row and column standard. There is always at least one way to do each of these two steps. Third, we need to fill the boxes of $F$ in $\lambda^{\prt{k}}\setminus (\mu^1\cup \mu^2)$ with the values $[|\mu^1|+1, |\lambda^{\prt{k}}|-|\mu^2|+1]$ in a way that is row and column standard. This may be achieved as follows.

Notice that $\lambda^{\prt{k}}\setminus (\mu^1\cup \mu^2)$ has the shape of two separate partitions, one in the first two rows of $\lambda^{\prt{k}}\setminus (\mu^1\cup \mu^2)$, and another in the first two columns. Denote these two partitions $\nu^1$ and $\nu^2$, respectively. For $(r,c)\in \nu^1$ or $(r,c)\in \nu^2$, let $(r,c)^{\circ}$ be the corresponding box in $\lambda^{\prt{k}}\setminus (\mu^1\cup \mu^2)$ and $\lambda^{\prt{k}}$. Partition $[|\mu^1|+1, |\lambda^{\prt{k}}|-|\mu^2|+1]$ into subsets $R_1$ and $R_2$ of size $|\nu^1|$ and $|\nu^2|$, respectively. Given a $Z_1 \in {\sf SYT}(\nu^1)$ we place the $i$th smallest element of $R_1$ into $F$ in $(r,c)^{\circ}$ where $(r,c)={\sf cell}_{Z_1}(i)$ in $\nu^1$. Given a $Z_2 \in {\sf SYT}(\nu^2)$ we place the $i$th smallest element of $R_2$ into $F$ in $(r,c)^{\circ}$ where $(r,c)={\sf cell}_{Z_2}(i)$ in $\nu^1$. Once this is done we have $F \in D(\lambda^{\prt{k}},\mu)$. 

The number of ways to achieve the third step is thus $\binom{|\nu^1|+|\nu^2|}{|\nu^1|}f^{\nu^1}f^{\nu^2}$. We conclude that
\begin{equation}
\label{eq:dlowerbound}
\binom{|\nu^1|+|\nu^2|}{|\nu^1|}f^{\nu^1}f^{\nu^2} \leq d(\lambda^{\prt{k}},\mu). 
\end{equation}


For each $(r,c)\in \nu^1$ and $(a, b)\in \nu^2$, $|h_{\nu^1}(r,c)| \leq |h_{\lambda^{\prt{k}}}((r,c)^{\circ})|$ and $|h_{\nu^2}(a, b)| \leq |h_{\lambda^{\prt{k}}}((a, b)^{\circ})|$. This implies
\begin{equation}
\label{eq:hooklengths}
\prod_{(r,c)\in \nu^1} (|h_{\nu^1}(r,c)| + 1) \prod_{(r, c)\in \nu^2} (|h_{\nu^2}(r, c)| + 1) \leq \prod_{(r, c)\in \lambda^{\prt{k}}\setminus (\mu^1\cup \mu^2)} (|h_{\lambda^{\prt{k}}}(r,c)| + 1)
\end{equation}
Combining the above arguments we have
\begin{align*}
\frac{f^{\lambda^{\prt{k}}}}{d(\lambda^{\prt{k}},\mu)} &\leq \frac{f^{\lambda^{\prt{k}}}}{f^{\nu^1}f^{\nu^2}\binom{|\nu^1|+|\nu^2|}{|\nu^1|}}\\
&= \frac{|\lambda^{\prt{k}}|!}{\prod_{(r,c)\in \lambda^{\prt{k}}} |h_{\lambda^{\prt{k}}}(r,c)|+1} \frac{\prod_{(r,c)\in \nu^1} (|h_{\nu^1}(r,c)| + 1) \prod_{(r, c)\in \nu^2} (|h_{\nu^2}(r, c)| + 1)}{(|\nu^1|+|\nu^2|)!}\\
&\leq \frac{|\lambda^{\prt{k}}|!}{\prod_{(r,c)\in \lambda^{\prt{k}}} (|h_{\lambda^{\prt{k}}}(r,c)|+1)} \frac{\prod_{(r,c)\in \lambda^{\prt{k}}\setminus (\mu^1\cup \mu^2)} (|h_{\lambda^{\prt{k}}}(r,c)|+1)}{(|\nu^1|+|\nu^2|)!}\\
&= \frac{|\lambda^{\prt{k}}|!}{(|\nu^1|+|\nu^2|)!\prod_{(r,c)\in \mu^1\cup\mu^2} (|h_{\lambda^{\prt{k}}}(r,c)|+1)}\\
&\leq  \frac{|\lambda^{\prt{k}}|!}{(|\nu^1|+|\nu^2|)!}\\
&=  \binom{|\lambda^{\prt{k}}|}{|\mu^1|+|\mu^2|}(|\mu^1|+|\mu^2|)!\\
&=  O(|\lambda^{\prt{k}}|^{|\mu^1|+|\mu^2|})O(1)\\
&={\sf poly}(|\lambda^{\prt{k}}|,k)
\end{align*}
where the first inequality is by~\eqref{eq:dlowerbound}, the second inequality is by~\eqref{eq:hooklengths}, and the last two equalities follow from the fact that $|\mu^1|+|\mu^2|\leq |\mu|+6 = \Theta(1)$.
\end{proof}

\noindent {\emph Case 2 ($ (\lambda^{\prt{k}})_1 > (\lambda^{\prt{k}})_2$, $ (\lambda^{\prt{k}})'_1 > (\lambda^{\prt{k}})'_2$, and either $(\lambda^{\prt{k}})_2 \leq p$ or $(\lambda^{\prt{k}})'_2  \leq q$):} In this case, the construction to lower bound $|{\sf SVT}(\lambda,N,S)|$ is almost identical to Case 1. The first difference is that the set $\mu^2$ will also include any $(r,c) \in\lambda^{\prt{k}}\setminus \delta^{{\sf sv}((\lambda^{\prt{k}})^+)-1}$ in rows or columns wholly contained in $\mu$. The second difference is that when inserting the values in $R_i$, if the values in $R_i$ would be inserted into a column or row wholly contained in $\mu$, then the values are instead inserted into $({\sf sv}((\lambda^{\prt{k}})^+)-i+1,i)$. Otherwise, proceeding exactly as in Case 1, we arrive at a lower bound on $|{\sf SVT}(\lambda,N,S)|$ which is then used to show that \eqref{eqn:polymixgen} is satisfied.

\noindent {\emph Case 3 ($ (\lambda^{\prt{k}})_1 = (\lambda^{\prt{k}})_2$ or $ (\lambda^{\prt{k}})'_1 = (\lambda^{\prt{k}})'_2$):} This case can be reduced to one of the first two cases by placing the value $k$ (and if needed $k-1$) into $S$ in the outermost box of the second row and/or second column. Say this augmented $S$ is $S^+$.  Since $|{\sf SVT}(\lambda,N,S^+)| \leq |{\sf SVT}(\lambda,N,S)|$, we can lower bound $|{\sf SVT}(\lambda,N,S)|$ by lower bounding $|{\sf SVT}(\lambda,N,S^+)|$. This can be achieved by applying Case 1 or 2 to $S^+$.
\end{proof}
Letting $p = q = 2$ we may prove the following two corollaries.
\begin{corollary}
\label{cor:rank2fpaus}
Let $\mathcal{R}_{\leq 2} \subseteq \mathcal{SVT}$ be the subset such that the rank of $\lambda$ is less than three.  Then $\mathcal{MC}_{\sf SVT}(\lambda,N,S)$ is rapidly mixing for all $(\lambda,N,S)\in \mathcal{R}_{\leq 2}$ and hence $\mathcal{MC}_{\sf SVT}(\lambda,N,S)$ is a FPAUS for $\mathcal{R}_{\leq 2}$.
\end{corollary}


To extend Theorem \ref{theorem:main} to all $(\lambda, N, S)$ would require devising a method for inserting the values of the $R_i$, from the proof of Theorem \ref{theorem:main}, in a row and column standard way when there are more than two arbitrarily long rows (or columns). We believe that this should be possible if the number of such rows and columns is upper bounded by a constant, which leads us to the following conjecture.


\begin{conjecture}
Let $d\in \mathbb{N}$ and $\mathcal{F}_d \subseteq \mathcal{SVT}$ be the subset such that ${\sf sv}((\lambda\setminus S)^+) \leq d$.  Then $\mathcal{MC}_{\sf SVT}(\lambda,N,S)$ is rapidly mixing for all $(\lambda,N,S)\in \mathcal{F}_d$.\end{conjecture}

\begin{theorem}
\label{thm:mainfixed}
Let ${\sf Fix}(|\lambda\setminus S|), {\sf Fix}(k-|\lambda\setminus S|) \subseteq \mathcal{SVT}$ be the subsets such that $|\lambda\setminus S|$ and $k-|\lambda\setminus S|$ are, respectively, $O(1)$. Then $\mathcal{MC}_{\sf SVT}(\lambda, N, S)$ is rapidly mixing, and hence $\mathcal{MC}_{\sf SVT}(\lambda,N,S)$ is a FPAUS, for ${\sf Fix}(|\lambda\setminus S|)$ and ${\sf Fix}(k-|\lambda\setminus S|)$.
\end{theorem}
\begin{proof}
Fix a $T\in {\sf SVT}(\lambda, N,S)$. Recall that $\lambda^{\prt{k}} := \lambda \setminus T^{\prt{k}} = \lambda \setminus S$. By \eqref{eqn:polymixgen}, it suffices to show that \[\frac{f^{\lambda^{\prt{k}}} \binom{k}{|\lambda^{\prt{k}}|}({\sf sv}((\lambda^{\prt{k}})^+)^{k-|\lambda^{\prt{k}}|}}{|{\sf SVT}(\lambda, N,S)|} = {\sf poly}(|\lambda^{\prt{k}}|, k).\]  

If $k = O(1)$, then since $|\lambda^{\prt{k}}|\leq k$, we have that $f^{\lambda^{\prt{k}}},\binom{k}{|\lambda^{\prt{k}}}, (R((\lambda^{\prt{k}})^+)^{k-|\lambda^{\prt{k}}|} = O(1)$. Thus even the lower bound $|{\sf SVT}(\lambda, N,S)| \geq 1$ gets us the desired growth bounds.

If $|\lambda^{\prt{k}}| = O(1)$, then $f^{\lambda^{\prt{k}}},{\sf sv}((\lambda^{\prt{k}})^+) = O(1)$ and $\binom{k}{|\lambda^{\prt{k}}|} = O((k)^{|\lambda^{\prt{k}}|})$.  We will construct elements $F \in {\sf SVT}(\lambda, N,S)$ as follows. Initialize $F$ to equal $S$. Then fill the cells of $F$ that are in $\delta^{{\sf sv}((\lambda^{\prt{k}})^+)} \cap \lambda^{\prt{k}}$ in any way with the smallest values in $[k]$, such that one value is in each cell and $F$ remains row and column standard. There is always at least one way to do this. Then, for the next $k-|\lambda^{\prt{k}}|$ labels, place each value in a cell of the form $(i,{\sf sv}((\lambda^{\prt{k}})^+)+1-i)$.  This process has $k-|\lambda^{\prt{k}}|$ values that each independently have ${\sf sv}((\lambda^{\prt{k}})^+)$ choices for the cell they are placed in. Hence $({\sf sv}((\lambda^{\prt{k}})^+)^{k-|\lambda^{\prt{k}}|} \leq |{\sf SVT}(\lambda, N,S)|$.  As a result, 
\begin{align*}
\frac{f^{\lambda^{\prt{k}}} \binom{k}{|\lambda^{\prt{k}}|}({\sf sv}((\lambda^{\prt{k}})^+)^{k-|\lambda^{\prt{k}}|}}{|{\sf SVT}(\lambda, N,S)|} &\leq \frac{f^{\lambda^{\prt{k}}} \binom{k}{|\lambda^{\prt{k}}|}({\sf sv}((\lambda^{\prt{k}})^+)^{k-|\lambda^{\prt{k}}|}}{{\sf sv}((\lambda^{\prt{k}})^+)^{k-|\lambda^{\prt{k}}|}}\\
&= f^{\lambda^{\prt{k}}} \binom{k}{|\lambda^{\prt{k}}|} = O((k)^{|\lambda^{\prt{k}}|})\\
&= {\sf poly}(|\lambda^{\prt{k}}|, k)
\end{align*}
completing the proof in this case.

If $k-|\lambda^{\prt{k}}| = O(1)$, then ${\sf sv}((\lambda^{\prt{k}})^+) = O(|\lambda^{\prt{k}}|)$ and ${\sf sv}((\lambda^{\prt{k}})^+)^{k-|\lambda^{\prt{k}}|} = {\sf poly}(|\lambda^{\prt{k}}|, k)$. Similarly, $\binom{k}{|\lambda^{\prt{k}}|} = O((k)^{k-|\lambda^{\prt{k}}|}) = {\sf poly}(|\lambda^{\prt{k}}|, k)$. We construct elements $F \in {\sf SVT}(\lambda, N,S)$ as follows. Initialize $F$ to equal $S$. Then place the values $1,\dots, |\lambda^{\prt{k}}|$ into $F$ such that $F$ remains column and row standard. From largest to smallest, place each value $|\lambda^{\prt{k}}| +1,\ldots, k$ in the leftmost cell of the first row of $F$ such that the value is the smallest element in that cell. By the hook-length formula, there are $f^{\lambda^{\prt{k}}}$ ways to arrange the first $|\lambda^{\prt{k}}|$ labels, and so $f^{\lambda^{\prt{k}}} \leq {\sf SVT}(\lambda, N,S)$. Thus 
\begin{align*}
\frac{f^{\lambda^{\prt{k}}} \binom{k}{|\lambda^{\prt{k}}|}({\sf sv}((\lambda^{\prt{k}})^+)^{k-|\lambda^{\prt{k}}|}}{|{\sf SVT}(\lambda, N,S)|} &\leq \frac{f^{\lambda^{\prt{k}}} \binom{k}{|\lambda^{\prt{k}}|}({\sf sv}((\lambda^{\prt{k}})^+)^{k-|\lambda^{\prt{k}}|}}{f^{\lambda^{\prt{k}}}}\\
&= \binom{k}{|\lambda^{\prt{k}}|}({\sf sv}((\lambda^{\prt{k}})^+)^{k-|\lambda^{\prt{k}}|}\\
&= O((k)^{k-|\lambda^{\prt{k}}|})O(|\lambda^{\prt{k}}|^{k-|\lambda^{\prt{k}}|})\\
&= {\sf poly}(|\lambda^{\prt{k}}|, k)
\end{align*}
completing the proof.
\end{proof}

\section{Approximate counting}
We are now ready to construct a FPRAS for the number of $N$-standard set-valued tableau that contain a fixed $N\prt{k}$-standard set-valued pre-tableau. Let $\lambda$ be a partition, $|\lambda| \leq N \in \mathbb{N}$, and $S$ a $N\prt{k}$-standard set-valued pre-tableau of shape $\lambda$ with $0\leq k \leq N$, error parameter $\epsilon \in (0,1]$ we desire to compute an approximation $A$ such that
\[
P((1 - \epsilon)f^{\lambda, N, S} \leq A \leq (1 + \epsilon)f^{\lambda, N, S}) \geq \frac{3}{4}
\]
in time polynomial in $|\lambda \setminus S|$, $k$, $\frac{1}{\epsilon}$. The confidence parameter of $\frac{3}{4}$ may be boosted to $1 - \delta$ for any $\delta > 0$ by performing $\mathcal{O}(\ln \delta^{-1})$ trials and taking the median result~\cite{JVV86}. By choosing $S = E_{\lambda, N}$ we may approximate $f^{\lambda, N}$.



Let $S_0, \ldots, S_k$ be a sequence such that, for $0 \leq m \leq k$, $S_m$ is a $N\prt{m}$-standard set-valued pre-tableaux and $S_0 \subseteq S_1 \subseteq \cdots \subseteq S_{k-1}, S_k = S$. Such a sequence must exist since $f^{\lambda, N, S} > 0$. Define $SV_m = {\sf SVT}(\lambda, N, S_m)$ for $0 \leq m \leq k$. Then
\begin{equation}
\label{eq:disjointSV}
SV_m = \bigsqcup_{A} {\sf SVT}(\lambda, N, A) 
\end{equation}
where the sum is over all $N\prt{m-1}$-standard set-valued pre-tableau $A$ such that $S_m \subseteq A$.


Then $f^{\lambda, N, S}$ can be computed via the telescoping product
\begin{equation}
\label{eq:flamNeq}
\frac{1}{f^{\lambda, N, S}} = \frac{1}{|SV_k|} = \frac{|SV_{k-1}|}{|SV_k|}\frac{|SV_{k-2}|}{|SV_{k-1}|}\cdots\frac{|SV_{0}|}{|SV_{1}|}.
\end{equation}
Note that $|SV_{0}|=1$ since $1$ may only be placed in the $(1,1)$ block, and $S_0$ will always be a $N$-standard set-valued tableau. Our FPRAS will approximate $f^{\lambda, N, S}$ by approximating the ratios $\frac{|SV_{m-1}|}{|SV_{m}|}$.

An $\mathcal{F} \subseteq \mathcal{SVT}$ is \emph{downwardly stable} if for all $(\lambda, N, S) \in \mathcal{F}$, if $S'$ is a $N\prt{k'}$-standard set-valued pre-tableau of shape $\lambda$ such that $S \subseteq S'$, then $(\lambda, N, S') \in \mathcal{F}$.

\begin{theorem}
\label{thm:downwardstabfrpas}
Let $\mathcal{F}$ be a downwardly stable subset of $\mathcal{SVT}$ such that there is a FPAUS for $\mathcal{F}$. There is a FPRAS that computes $f^{\lambda, N, S}$ for $(\lambda, N, S) \in \mathcal{F}$. 
\end{theorem}
\begin{proof}
Let $(\lambda, N, S) \in \mathcal{F}$ and set $S_k = S$. We will approximate the ratios from \eqref{eq:flamNeq} for $m = k, k-1, \ldots, 1$, inductively, by sampling almost uniformly at random from $SV_{m}$ using the FPAUS. We run the FPAUS, with bias parameter $\eta = \frac{\epsilon}{20 |\lambda \setminus S| k}$\,, to sample $s$ samples from $SV_{m} = {\sf SVT}(\lambda, N, S_m)$ with distribution $\pi$ such that $d_{\sf TV}(\pi,u) \leq \eta$.
By \eqref{eq:disjointSV}, each sample is contained in a ${\sf SVT}(\lambda, N, A)$ for $A$ a $N\prt{m-1}$-standard set-valued pre-tableau $A$ such that $S_m \subseteq A$. Let $S_{m-1}$ be the $N\prt{m-1}$-standard set-valued pre-tableau such that subset ${\sf SVT}(\lambda, N, S_{m-1})$ contains the largest number of samples. The set $\mathcal{F}$ is downwardly stable, and hence $(\lambda, N, S_{m-1}) \in \mathcal{F}$. Thus, we construct the sequence $S = S_k \subseteq S_{k-1} \subseteq \cdots \subseteq S_{1} \subseteq S_0$ of pre-tableau with $SV_m = {\sf SVT}(\lambda, N, S_m)$ for $0 \leq m \leq k$. If
\[
r_m = \frac{|SV_{m-1}|}{|SV_{m}|},
\]
then, $d_{\sf TV}(\pi,u) \leq \eta$ implies
\begin{equation}
\label{eq:rel1}
r_m - \eta = u(SV_{m-1}) - \eta \leq \pi(SV_{m-1}) \leq u(SV_{m-1}) + \eta = r_m + \eta.
\end{equation}

Our aim is to estimate $r_m$ within a multiplicative factor of $(1 + \epsilon / 4 k)$ with probability $1 - 1 / 4k$. Let $X_i$, for $1 \leq i \leq s$, be a random variable equal to $1$ if sample $i$ is in $SV_{m-1}$, and $0$ otherwise. Let $\sigma^2 = Var(X_1) = \ldots = Var(X_s)$.
Let $X^{[m]} = \sum_{i=1}^s X_i$. Note that the choice of $S_{k-1}$ and \eqref{eq:disjointSV} ensures that $1/|\lambda \setminus S| \leq X^{[m]}/s$. The Chebyshev inequality~\cite[Theorem 4.1.1]{AS06} implies
\[
\mathbb{P}(E_1):=\mathbb{P}(|X^{[m]}/s - \pi(SV_{m-1})| \geq \eta) \leq \frac{\sigma^2}{s \eta^2} \leq \frac{1}{s \eta^2}.
\]
Choosing $s \geq 8k(\frac{20 |\lambda \setminus S| k}{\epsilon})^2$, we have that with probability greater than $1 - \frac{1}{8k}$ the event $\overline{E_1}$ (the complement of event $E_1$) occurs and
\begin{equation}
\label{eq:prob1}
\qquad\qquad \frac{1}{|\lambda \setminus S|} \leq X^{[m]}/s  \leq \,\, \pi(SV_{m-1}) + \eta \leq r_m + 2\eta.  \\
\end{equation}

Applying a Chernoff bound~\cite[Corollary 4.5]{MU05}, with $0 < \eta < 1$, yields
\[
\mathbb{P}(|X^{[m]} - E[X^{[m]}]| \geq \eta E[X^{[m]}]) \leq 2e^{-\eta^2 E[X^{[m]}]/3}.
\]
Given that the event $\overline{E_1}$ occurs and choosing $s \geq \max\{8k(\frac{20 |\lambda \setminus S| k}{\epsilon})^2, 3(\frac{20 |\lambda \setminus S| k}{\epsilon})^2 \frac{|\lambda \setminus S|}{1-\epsilon} \log(16k)\}$ by \eqref{eq:prob1}
\[
2e^{-\eta^2 E[X^{[m]}]/3} \leq 2e^{-\eta^2 s (\frac{1}{|\lambda \setminus S|} - \eta) /3} \leq 2e^{-\eta^2 s (\frac{1-\epsilon}{|\lambda \setminus S|}) /3} \leq \frac{1}{8k}.
\]
Thus the probability of $\overline{E_1}$ and $|X^{[m]} - E[X^{[m]}]| < \eta E[X^{[m]}]$ is greater than $1 - \frac{1}{4k}$. Hence with probability greater than $1 - \frac{1}{4k}$, we have
\begin{align*}
X^{[m]}  / s & \leq \,\, (1+\eta)E[X^{[m]}] / s = (1+\eta)\pi(SV_{m-1}) & \\
& \leq \,\, (1+\eta)(r_m + \eta) & \text{By }\eqref{eq:rel1}. \\
& = \,\, r_m (1 + \eta + \eta / r_m + \eta^2 / r_m) \leq \,\, r_m (1 + \eta + 2\eta / r_m) & \\
& \leq \,\, r_m (1 + \eta + 2\eta / (\frac{1}{|\lambda \setminus S|}-2\eta)) & \text{By }\eqref{eq:prob1}. \\
& \leq \,\, r_m (1 + 5|\lambda \setminus S|\eta) = r_m (1 + \frac{\epsilon}{4k}). & 
\end{align*}
The third inequality above follows from $\eta < 1$ and the final inequality from the fact that $2 \eta \leq \frac{1}{2 |\lambda \setminus S|}$. By a nearly identical argument we have $r_m (1 - \frac{\epsilon}{4k}) \leq X^{[m]} / s$. Now, multiplying $\epsilon$ by a sufficiently small constant if needed (that does not depend on $\epsilon$), we have that $1/(1+\epsilon/4k)^k \geq (1 - \epsilon/4k)^k \geq (1 - \epsilon)$ and $1/(1-\epsilon/4k)^k \leq (1 + \epsilon/2k)^k \leq (1 + \epsilon)$. This, combined with the above arguments and \eqref{eq:flamNeq}, implies that if $A = \prod_{m=1}^k X^{[m]} / s$,
\[
\mathbb{P}((1 - \epsilon)f^{\lambda, N, S} \leq 1/A \leq (1 + \epsilon)f^{\lambda, N, S}) \geq \frac{3}{4}.
\]

The FPAUS for each $SV_m$ is polynomial in $|\lambda \setminus S_m|$, and $m$ by hypothesis and hence are polynomial in $|\lambda \setminus S_k|$, and $k$. A total of $k \cdot \max\{8k(\frac{20 |\lambda \setminus S| k}{\epsilon})^2, 3(\frac{20 |\lambda \setminus S| k}{\epsilon})^2 \frac{|\lambda \setminus S|}{1-\epsilon} \log(8k)\}$ samples are required from the FPAUS. Thus our approximation is computed in time polynomial in $|\lambda \setminus S|$, $k$, $\frac{1}{\epsilon}$.
\end{proof}

\begin{corollary}
There is a FPRAS computing $f^{\lambda, N, S}$ for $\mathcal{F}_{p,q}$.
\end{corollary}
\begin{proof}
The subset $\mathcal{F}_{p,q} \subseteq \mathcal{SVT}$ is downwardly stable. Our result now follows from Theorem~\ref{theorem:main} and Theorem~\ref{thm:downwardstabfrpas}.
\end{proof}

\begin{corollary}
There is a FPRAS computing $f^{\lambda, N, S}$ for ${\sf Fix}(|\lambda\setminus S|)$ and ${\sf Fix}(k-|\lambda\setminus S|)$.
\end{corollary}
\begin{proof}
The subsets ${\sf Fix}(|\lambda\setminus S|), {\sf Fix}(k-|\lambda\setminus S|) \subseteq \mathcal{SVT}$ are each downwardly stable. Our result follows from Theorem~\ref{thm:mainfixed} and Theorem~\ref{thm:downwardstabfrpas}.
\end{proof}

We conclude with proofs of our main theorems.

\noindent \textit{Proof of Theorem~\ref{thm:mainFPRASresult}:} The set of $(\lambda,N,E_{\lambda, N})$ where $\lambda$ is a partition such that $\lambda \subseteq \mu \cup \lambda^{\circ}$ and $\lambda^{\circ}$ is a partition of rank less than three is a subset of $\mathcal{F}_{p,q}$. The existence of a FPRAS for $\mathcal{F}_{p,q}$ yields a FPRAS for this subset. \qed

\noindent \textit{Proof of Theorem~\ref{thm:mainFPRASresultfixed}:} This follows by an identical argument to the proof of Theorem~\ref{thm:mainFPRASresult}, applying the existence of a FPRAS for ${\sf Fix}(|\lambda\setminus S|)$ and ${\sf Fix}(k-|\lambda\setminus S|)$.

\section*{Acknowledgements}
We indebted to Alex Yong for suggesting this problem and for many helpful discussions. 
We are also grateful to Alejandro Morales for helpful comments and suggestions.
This research was partially supported by NSF RTG grant DMS 1937241.
RH was partially supported by an AMS Simons Travel grant.

\end{document}